\documentclass[12pt,twoside]{amsart}
\usepackage{amsmath, amsthm, amscd, amsfonts, amssymb, graphicx}
\usepackage{enumerate}
\usepackage[colorlinks=true,
linkcolor=blue,
urlcolor=cyan,
citecolor=red]{hyperref}
\usepackage{mathrsfs}
\newtheorem{theorem}{Theorem}[section]
\newtheorem{lemma}{Lemma}[section]

\newtheorem{corollary}{Corollary}[section]

\newtheorem{proposition}{Proposition}[section]
\numberwithin{equation}{section}

\begin{document}
	
\title{More accurate  numerical radius inequalities}
\author{Mohammad Sababheh and Hamid Reza Moradi }
\maketitle

\begin{abstract}
In this article, we present some new general forms of numerical radius inequalities for Hilbert space operators. The significance of these inequalities follow from the way they extend and refine some known results in this field. Among other inequalities, it is shown that if $A$ is a bounded linear operator on a complex Hilbert space, then
\[{{w}^{2}}\left( A \right)\le \left\| \int_{0}^{1}{{{\left( t\left| A \right|+\left( 1-t \right)\left| {{A}^{*}} \right| \right)}^{2}}dt} \right\|\le \frac{1}{2}\left\| \;{{\left| A \right|}^{2}}+{{\left| {{A}^{*}} \right|}^{2}} \right\|\]
where $w\left( A \right)$ and $\left\| A \right\|$ are the numerical radius and the usual operator norm of $A$, respectively. 
\end{abstract}


\section{Introduction}
Let $\mathcal{B}(\mathcal{H})$ denote the $C^*$ algebra of all bounded linear operators on a complex Hilbert space $\mathcal{H}.$ An operator $A\in\mathcal{B}(\mathcal{H})$ is associated with some numerical inequalities to better understand these operators and to facilitate the comparison between such operators.

In this article, we investigate the (usual) operator norm $\|A\|$ and the numerical radius $w(A)$ of the operator $A$, defined respectively by
$$\|A\|=\sup_{\|x\|=1}\|Ax\|\;{\text{and}}\; w(A)=\sup_{\|x\|=1}|\left<Ax,x\right>|;\;x\in\mathcal{H}.$$
These quantities define equivalent norms on $\mathcal{B}(\mathcal{H})$ as one has the equivalence inequalities
\begin{equation}\label{equiv_ineq_intro}
\frac{1}{2}\|A\|\leq w(A)\leq \|A\|;\;A\in\mathcal{B}(\mathcal{H}).
\end{equation}

Such inequalities are important as one can have upper or lower bounds of one quantity in terms of the other. Consequently, having sharper bounds is highly demanded in this field. We refer the reader to \cite{golla,gau,drag_1,3,1,moradi1,sab,shabrawi} as a sample of treatments of this interest.

In \cite{1}, Kittaneh proved that for any $A\in \mathcal{B}\left( \mathcal{H} \right)$,
\begin{equation}\label{kitt_ineq_intro}
{{w}^{2}}\left( A \right)\le \frac{1}{2}\left\| {{\left| A \right|}^{2}}+{{\left| {{A}^{*}} \right|}^{2}} \right\|.
\end{equation}
Here $\left| A \right|$  stands for the positive  operator ${{\left( {{A}^{*}}A \right)}^{\frac{1}{2}}}$. This inequality is sharper than the right inequality in \eqref{equiv_ineq_intro} since $\left\| {{\left| A \right|}^{2}}+{{\left| {{A}^{*}} \right|}^{2}} \right\|\le 2{{\left\| A \right\|}^{2}}$.

The main goal of this article is two folded. First, we show that \eqref{kitt_ineq_intro} follows from a more general convex form and, second, to present refined versions of   \eqref{kitt_ineq_intro} and the general form as well.

More precisely, we show that if $f:[0,\infty)\to \mathbb{R}$ is an increasing operator convex function, then for $A\in\mathcal{B}(\mathcal{H})$,
\begin{equation}\label{first_ineq_intro}
f\left( w\left( A \right) \right)\le \left\| \int_{0}^{1}{f\left( t\left| A \right|+\left( 1-t \right)\left| {{A}^{*}} \right| \right)dt} \right\|\le \frac{1}{2}\left\| f\left( \left| A \right| \right)+f\left( \left| {{A}^{*}} \right| \right) \right\|.
\end{equation}

Then upon selecting certain functions, we obtain specific inequalities. For example, letting $f(t)=t^2$ implies and refines \eqref{kitt_ineq_intro}.

With this theme, we present several new convex inequalities.

Our arguments make use of the well known Hermite-Hadamard inequalities stating that for a convex function $f:J\to \mathbb{R}$, and for any $a,b\in J$ we have 
\begin{equation}\label{13}
f\left( \frac{a+b}{2} \right)\le \int_{0}^{1}{f\left( ta+\left( 1-t \right)b \right)dt}\le \frac{f\left( a \right)+f\left( b \right)}{2}.
\end{equation}
The following two well known inequalities for Hilbert space operators are also useful to accomplish our results. 
\begin{lemma}\label{14} (see \cite{kitt2})
Let $A\in \mathcal{B}\left( \mathcal{H} \right)$ and let  $x,y\in \mathcal{H}$ be any vectors. Then
\[\left| \left\langle Ax,y \right\rangle  \right|\le \sqrt{\left\langle \left| A \right|x,x \right\rangle \left\langle \left| {{A}^{*}} \right|y,y \right\rangle },\]
where $A^*$ is the adjoint operator of $A, |A|=(A^{*}A)^{1/2}$ and $|A^{*}|=(AA^*)^{1/2}.$
\end{lemma}
\begin{lemma}\label{15}
Let $A\in \mathcal{B}\left( \mathcal{H} \right)$ be a self-adjoint operator with  spectrum contained in the interval $J$, and let $x\in \mathcal{H}$ be a unit vector. If $f$ is a convex function on $J$, then	
\[f\left( \left\langle Ax,x \right\rangle  \right)\le \left\langle f\left( A \right)x,x \right\rangle.\]
\end{lemma}

We will also need to recall the definition of an operator convex function. Namely, a function $f:J\to \mathbb{R}$ is said to be an operator convex function if $f$ is continuous and $f\left(\frac{A+B}{2}\right)\leq \frac{f(A)+f(B)}{2}$ for all self adjoint operators $A,B\in\mathcal{B}(\mathcal{H})$ with spectra in the interval $J.$

Other results will involve numerical radius inequalities with unital positive linear mappings. Recall that a linear map $\Phi:\mathcal{B}(\mathcal{H})\to\mathcal{B}(\mathcal{K})$ is said to be positive if it maps positive operators to positive operators. That is, if $A\geq 0$ in $\mathcal{B}(\mathcal{H})$, then $\Phi(A)\geq 0$ in $\mathcal{B}(\mathcal{K}).$ If in addition $\Phi(I_{\mathcal{H}})=I_{\mathcal{K}}$ then $\Phi$ is said to be unital. In this context, $I_{\mathcal{H}}$ is the identity mapping on $\mathcal{H}.$

Among the most interesting properties of such maps is the well known Choi-Davis inequality that states \cite{choi}
\begin{equation}\label{choi-davis-intro}
f\left(\Phi(A)\right)\leq \Phi(f(A)),
\end{equation}
for the unital positive linear map $\Phi,$ the operator convex function $f:J\to\mathbb{R}$ and the self adjoint operator $A$ whose spectrum is in the interval $J$. As a consequence, for example, one has the Kadison inequality
\begin{equation}\label{kadison_intro}
\Phi^2(A)\leq \Phi(A^2).
\end{equation}

It is unfortunate that the Inequality \eqref{choi-davis-intro} is not valid for convex functions $f$. However, if $f$ is convex, instead of operator convex, then one has the weaker inequality (see \cite[lemma 2.1]{mor})
\begin{equation}\label{conv_inner_phi_intro}
f\left(\left<\Phi(A)x,x\right>\right)\leq \left<\Phi(f(A))x,x\right>;x\in\mathcal{H},\;\|x\|=1.
\end{equation}

\section{Main Results}
In this section we prove our results for convex functions, then we show how these results are related to the existed ones.
\begin{theorem}\label{thm_1}
Let $A\in \mathcal{B}\left( \mathcal{H} \right)$. If $f:[0,\infty)\to [0,\infty)$ is an increasing operator convex function, then
\begin{equation}\label{18}
f\left( w\left( A \right) \right)\le \left\| \int_{0}^{1}{f\left( t\left| A \right|+\left( 1-t \right)\left| {{A}^{*}} \right| \right)dt} \right\|\le \frac{1}{2}\left\| f\left( \left| A \right| \right)+f\left( \left| {{A}^{*}} \right| \right) \right\|.
\end{equation}
\end{theorem}
\begin{proof}
Let $x\in \mathcal{H}$ be a unit vector. Then 
\begin{equation}\label{16}
\begin{aligned}
f\left( \left| \left\langle Ax,x \right\rangle  \right| \right)&\le f\left( \sqrt{\left\langle \left| A \right|x,x \right\rangle \left\langle \left| {{A}^{*}} \right|x,x \right\rangle } \right) \quad \text{(by Lemma \ref{14})}\\ 
& \le f\left( \frac{\left\langle \left| A \right|x,x \right\rangle +\left\langle \left| {{A}^{*}} \right|x,x \right\rangle }{2} \right) \quad \text{(by AM-GM inequality)}\\ 
& \le \int_{0}^{1}{f\left( t\left\langle \left| A \right|x,x \right\rangle +\left( 1-t \right)\left\langle \left| {{A}^{*}} \right|x,x \right\rangle  \right)dt}\quad \text{(by \eqref{13})}.  
\end{aligned}
\end{equation}
On the other hand,
\[\begin{aligned}
 f\left( t\left\langle \left| A \right|x,x \right\rangle +\left( 1-t \right)\left\langle \left| {{A}^{*}} \right|x,x \right\rangle  \right)&=f\left( \left\langle \left(t\left| A \right|+\left( 1-t \right)\left| {{A}^{*}} \right|\right)x,x \right\rangle  \right) \\ 
& \le \left\langle f\left( t\left| A \right|+\left( 1-t \right)\left| {{A}^{*}} \right| \right)x,x \right\rangle  \quad \text{(by Lemma \ref{15})}\\ 
& \le t\left\langle f\left( \left| A \right| \right)x,x \right\rangle +\left( 1-t \right)\left\langle f\left( \left| {{A}^{*}} \right| \right)x,x \right\rangle ,
\end{aligned}\]
where the last inequality follows from operator convexity of $f$.
Now, integrating over $t$ on $\left[ 0,1 \right]$,
\begin{equation}\label{17}
\begin{aligned}
\int_{0}^{1}{f\left( t\left\langle \left| A \right|x,x \right\rangle +\left( 1-t \right)\left\langle \left| {{A}^{*}} \right|x,x \right\rangle  \right)}&\le \left\langle \int_{0}^{1}{f\left( t\left| A \right|+\left( 1-t \right)\left| {{A}^{*}} \right| \right)dt}x,x \right\rangle  \\ 
& \le \frac{1}{2}\left\langle f\left(\left( \left| A \right| \right)+f\left( \left| {{A}^{*}} \right| \right)\right)x,x \right\rangle .  
\end{aligned}
\end{equation}
Combining  \eqref{16} and \eqref{17}, we get
\[f\left( \left| \left\langle Ax,x \right\rangle  \right| \right)\le \left\langle \int_{0}^{1}\left\{f\left( t\left| A \right|+\left( 1-t \right)\left| {{A}^{*}} \right| \right)dt\right\}x,x \right\rangle \le \frac{1}{2}\left\langle \left(f\left( \left| A \right| \right)+f\left( \left| {{A}^{*}} \right| \right)\right)x,x \right\rangle .\]
By taking supremum we get the desired result \eqref{18}.
\end{proof}
Since the function $f(t)=t^r, 1\leq r\leq 2$ is an increasing operator convex function, Theorem \ref{thm_1} implies the following general form of Kittaneh inequality \eqref{kitt_ineq_intro}.
\begin{corollary}
Let $A\in \mathcal{B}\left( \mathcal{H} \right)$. Then for any $1\le r\le 2$,
\[{{w}^{r}}\left( A \right)\le \left\| \int_{0}^{1}{{{\left( t\left| A \right|+\left( 1-t \right)\left| {{A}^{*}} \right| \right)}^{r}}dt} \right\|\le \frac{1}{2}\left\| \;{{\left| A \right|}^{r}}+{{\left| {{A}^{*}} \right|}^{r}} \right\|.\]
In particular,
\[{{w}^{2}}\left( A \right)\le \left\| \int_{0}^{1}{{{\left( t\left| A \right|+\left( 1-t \right)\left| {{A}^{*}} \right| \right)}^{2}}dt} \right\|\le \frac{1}{2}\left\| \;{{\left| A \right|}^{2}}+{{\left| {{A}^{*}} \right|}^{2}} \right\|.\]
\end{corollary}

Notice that the condition that $f$ is operator convex is essential to accomplish the proof of Theorem \ref{thm_1}. In the next result, we present another version for convex functions. 

\begin{theorem}\label{7}
Let $\Phi :\mathcal{B}\left( \mathcal{H} \right)\to \mathcal{B}\left( \mathcal{H} \right)$ be a unital positive linear map and let $A\in \mathcal{B}\left( \mathcal{H} \right)$. If $f:[0,\infty)\to [0,\infty)$ is an increasing convex function, then
\[\begin{aligned}
 f\left( {{w}^{2}}\left( \Phi \left( A \right) \right) \right)&\le \underset{\left\| x \right\|=1}{\mathop{\underset{x\in \mathcal{H}}{\mathop{\sup }}\,}}\,\left( \int_{0}^{1}{f\left( {{\left\| {{\Phi }^{\frac{1}{2}}}\left( t{{\left| A \right|}^{2}}+\left( 1-t \right){{\left| {{A}^{*}} \right|}^{2}} \right)x \right\|}^{2}} \right)dt} \right) \\ 
& \le \frac{1}{2}\left\| \Phi \left( f\left( {{\left| A \right|}^{2}} \right)+f\left( {{\left| {{A}^{*}} \right|}^{2}} \right) \right) \right\|.  
\end{aligned}\]
\end{theorem}
\begin{proof}
Let $A=B+iC$  be the Cartesian decomposition of $A\in\mathcal{B}(\mathcal{H})$. That is,  $B=\frac{A+{{A}^{*}}}{2}$ and $C=\frac{A-{{A}^{*}}}{2i}$.  Then clearly,
\begin{equation}\label{4}
\frac{{{A}^{*}}A+A{{A}^{*}}}{2}={{B}^{2}}+{{C}^{2}},	
\end{equation}
and
\begin{equation}\label{6}
{{\left| \left\langle Ax,x \right\rangle  \right|}^{2}}={{\left\langle Bx,x \right\rangle }^{2}}+{{\left\langle Cx,x \right\rangle }^{2}},~\text{ }x\in \mathcal{H},~\left\| x \right\|=1.
\end{equation}
Replacing $a$ and $b$ by $\left\langle \Phi \left( {{A}^{*}}A \right)x,x \right\rangle $ and $\left\langle \Phi \left( A{{A}^{*}} \right)x,x \right\rangle $, where $x\in \mathcal{H}$ is a unit vector, in \eqref{13} we obtain
\begin{equation}\label{1}
\begin{aligned}
f\left( \frac{\left\langle \Phi \left( {{A}^{*}}A \right)x,x \right\rangle +\left\langle \Phi \left( A{{A}^{*}} \right)x,x \right\rangle }{2} \right)&\le \int_{0}^{1}{f\left( t\left\langle \Phi \left( {{A}^{*}}A \right)x,x \right\rangle +\left( 1-t \right)\left\langle \Phi \left( A{{A}^{*}} \right)x,x \right\rangle  \right)} \\ 
& \le \frac{f\left( \left\langle \Phi \left( {{A}^{*}}A \right)x,x \right\rangle  \right)+f\left( \left\langle \Phi \left( A{{A}^{*}} \right)x,x \right\rangle  \right)}{2}.  
\end{aligned}
\end{equation}
Noting \eqref{conv_inner_phi_intro}, we have
\begin{equation}\label{2}
\begin{aligned}
\frac{f\left( \left\langle \Phi \left( {{A}^{*}}A \right)x,x \right\rangle  \right)+f\left( \left\langle \Phi \left( A{{A}^{*}} \right)x,x \right\rangle  \right)}{2}&\le \frac{\left\langle \Phi \left( f\left( {{A}^{*}}A \right) \right)x,x \right\rangle +\left\langle \Phi \left( f\left( A{{A}^{*}} \right) \right)x,x \right\rangle }{2} \\ 
& = \frac{1}{2}\left\langle \Phi \left( f\left( {{A}^{*}}A \right) + f\left( A{{A}^{*}} \right) \right)x,x \right\rangle.   
\end{aligned}
\end{equation}
By combining the second inequality in \eqref{1} and \eqref{2}, then taking the supremum, we have
\begin{equation}\label{3}
\underset{\left\| x \right\|=1}{\mathop{\underset{x\in \mathcal{H}}{\mathop{\sup }}\,}}\,\left( \int_{0}^{1}{f\left( \left\langle \Phi \left( t{{A}^{*}}A+\left( 1-t \right)A{{A}^{*}} \right)x,x \right\rangle  \right)dt} \right)\le \frac{1}{2}\left\| \Phi \left( f\left( {{A}^{*}}A \right)+f\left( A{{A}^{*}} \right) \right) \right\|.
\end{equation}
On the other hand, noting that $f$ is increasing,
\begin{equation}\label{9}
\begin{aligned}
 f\left( \frac{\left\langle \Phi \left( {{A}^{*}}A \right)x,x \right\rangle +\left\langle \Phi \left( A{{A}^{*}} \right)x,x \right\rangle }{2} \right)&=f\left( \frac{\left\langle \Phi \left( {{A}^{*}}A+A{{A}^{*}} \right)x,x \right\rangle }{2} \right) \\ 
& =f\left( \left\langle \Phi \left( {{B}^{2}}+{{C}^{2}} \right)x,x \right\rangle  \right)\;({\text{by}}\;\eqref{4}) \\ 
& =f\left( \left\langle \Phi \left( {{B}^{2}} \right)x,x \right\rangle +\left\langle \Phi \left( {{C}^{2}} \right)x,x \right\rangle  \right) \\ 
& \geq f\left( \left\langle {{\Phi }^{2}}\left( B \right)x,x \right\rangle +\left\langle {{\Phi }^{2}}\left( C \right)x,x \right\rangle  \right)\;({\text{by}}\;\eqref{kadison_intro}) \\ 
& \ge f\left( {{\left\langle \Phi \left( B \right)x,x \right\rangle }^{2}}+{{\left\langle \Phi \left( C \right)x,x \right\rangle }^{2}} \right)\;({\text{by}}\;\eqref{conv_inner_phi_intro}) \\ 
& =f\left( {{\left| \left\langle \Phi \left( A \right)x,x \right\rangle  \right|}^{2}} \right)\;({\text{by}}\;\eqref{6}).  
\end{aligned}
\end{equation}
Consequently, by combining the first inequality in \eqref{1} and \eqref{9}, then taking the supremum, we get
\begin{equation}\label{5}
f\left( {{w}^{2}}\left( \Phi \left( A \right) \right) \right)\le \underset{\left\| x \right\|=1}{\mathop{\underset{x\in \mathcal{H}}{\mathop{\sup }}\,}}\,\left( \int_{0}^{1}{f\left( \left\langle \Phi \left( t{{A}^{*}}A+\left( 1-t \right)A{{A}^{*}} \right)x,x \right\rangle  \right)dt} \right).
\end{equation}
Finally,  \eqref{5} and \eqref{3} imply the desired result.
\end{proof}
\begin{corollary}
Let $\Phi :\mathcal{B}\left( \mathcal{H} \right)\to \mathcal{B}\left( \mathcal{H} \right)$ be a unital positive linear map and let $A\in \mathcal{B}\left( \mathcal{H} \right)$. Then for any $r\ge 1$,
\[{{w}^{2r}}\left( \Phi \left( A \right) \right)\le \underset{\left\| x \right\|=1}{\mathop{\underset{x\in \mathcal{H}}{\mathop{\sup }}\,}}\,\left( \int_{0}^{1}{{{\left\| {{\Phi }^{\frac{1}{2}}}\left( t{{\left| A \right|}^{2}}+\left( 1-t \right){{\left| {{A}^{*}} \right|}^{2}} \right)x \right\|}^{2r}}dt} \right)\le \frac{1}{2}\left\| \Phi \left( {{\left| A \right|}^{2r}}+{{\left| {{A}^{*}} \right|}^{2r}} \right) \right\|.\]
and
\[{{w}^{2r}}\left( A \right)\le \underset{\left\| x \right\|=1}{\mathop{\underset{x\in \mathcal{H}}{\mathop{\sup }}\,}}\,\left( \int_{0}^{1}{{{\left\| {{\left( t{{\left| A \right|}^{2}}+\left( 1-t \right){{\left| {{A}^{*}} \right|}^{2}} \right)}^{\frac{1}{2}}}x \right\|}^{2r}}dt} \right)\le \frac{1}{2}\left\| {{\left| A \right|}^{2r}}+{{\left| {{A}^{*}} \right|}^{2r}} \right\|.\]
In particular,
\[{{w}^{2}}\left( A \right)\le \underset{\left\| x \right\|=1}{\mathop{\underset{x\in \mathcal{H}}{\mathop{\sup }}\,}}\,\left( \int_{0}^{1}{{{\left\| {{\left( t{{\left| A \right|}^{2}}+\left( 1-t \right){{\left| {{A}^{*}} \right|}^{2}} \right)}^{\frac{1}{2}}}x \right\|}^{2}}dt} \right)\le \frac{1}{2}\left\| {{\left| A \right|}^{2}}+{{\left| {{A}^{*}} \right|}^{2}} \right\|.\]
\end{corollary}

Assume that ${{A}_{1}},\ldots ,{{A}_{n}}$  are operators and ${{\Phi }_{1}},\ldots ,{{\Phi }_{n}}$  are positive linear maps  with $\sum\nolimits_{i=1}^{n}{{{\Phi }_{i}}\left( I \right)}=I$. For $A,B\in \mathcal{B}\left( \mathcal{H} \right)$ assume that $A\oplus B$ is the operator defined on $\mathcal{B}\left( \mathcal{H}\oplus \mathcal{H} \right)$ by $\left( \begin{matrix}
A & 0  \\
0 & B  \\
\end{matrix} \right)$. Now if we apply Theorem \ref{7} to the operator $A$ on the Hilbert space $\mathcal{H}\oplus \cdots \oplus \mathcal{H}$ defined by $A={{A}_{1}}\oplus \cdots \oplus {{A}_{n}}$ and the positive linear map $\Phi $ defined on $\mathcal{B}\left( \mathcal{H}\oplus \cdots \oplus \mathcal{H} \right)$ by $\Phi \left( A \right)={{\Phi }_{1}}\left( {{A}_{1}} \right)\oplus \cdots \oplus {{\Phi }_{n}}\left( {{A}_{n}} \right)$, then
\[\begin{aligned}
 f\left( {{w}^{2}}\left( \sum\limits_{i=1}^{n}{{{\Phi }_{i}}\left( {{A}_{i}} \right)} \right) \right)&\le \underset{\left\| x \right\|=1}{\mathop{\underset{x\in \mathcal{H}}{\mathop{\sup }}\,}}\,\left( \int_{0}^{1}{f\left( {{\left\| \sum\limits_{i=1}^{n}{\Phi _{i}^{\frac{1}{2}}\left( t{{\left| {{A}_{i}} \right|}^{2}}+\left( 1-t \right){{\left| A_{i}^{*} \right|}^{2}} \right)x} \right\|}^{2}} \right)dt} \right) \\ 
& \le \frac{1}{2}\left\| \sum\limits_{i=1}^{n}{{{\Phi }_{i}}\left( f\left( {{\left| {{A}_{i}} \right|}^{2}} \right)+f\left( {{\left| A_{i}^{*} \right|}^{2}} \right) \right)} \right\|.  
\end{aligned}\]
In particular,
\[\begin{aligned}
 f\left( {{w}^{2}}\left( \sum\limits_{i=1}^{n}{P_{i}^{*}{{A}_{i}}{{P}_{i}}} \right) \right)&\le \underset{\left\| x \right\|=1}{\mathop{\underset{x\in \mathcal{H}}{\mathop{\sup }}\,}}\,\left( \int_{0}^{1}{f\left( {{\left\| \sum\limits_{i=1}^{n}{{{\left( P_{i}^{*}\left( t{{\left| {{A}_{i}} \right|}^{2}}+\left( 1-t \right){{\left| A_{i}^{*} \right|}^{2}} \right){{P}_{i}} \right)}^{\frac{1}{2}}}x} \right\|}^{2}} \right)dt} \right) \\ 
& \le \frac{1}{2}\left\| \sum\limits_{i=1}^{n}{P_{i}^{*}\left( f\left( {{\left| {{A}_{i}} \right|}^{2}} \right)+f\left( {{\left| A_{i}^{*} \right|}^{2}} \right) \right){{P}_{i}}} \right\|  
\end{aligned}\]
where ${{P}_{1}},\ldots {{P}_{n}}$ are contractions with $\sum\nolimits_{i=1}^{n}{P_{i}^{*}{{P}_{i}}}=I$. If ${{A}_{1}},\ldots ,{{A}_{n}}$  are normal operators, then
\[f\left( {{\left\| \sum\limits_{i=1}^{n}{P_{i}^{*}{{A}_{i}}{{P}_{i}}} \right\|}^{2}} \right)\le \underset{\left\| x \right\|=1}{\mathop{\underset{x\in \mathcal{H}}{\mathop{\sup }}\,}}\,\left( f\left( {{\left\| \sum\limits_{i=1}^{n}{{{\left( P_{i}^{*}{{\left| {{A}_{i}} \right|}^{2}}{{P}_{i}} \right)}^{\frac{1}{2}}}x} \right\|}^{2}} \right) \right)\le \left\| \sum\limits_{i=1}^{n}{P_{i}^{*}f\left( {{\left| {{A}_{i}} \right|}^{2}} \right){{P}_{i}}} \right\|.\]
\begin{proposition}\label{8}
Let $\Phi :\mathcal{B}\left( \mathcal{H} \right)\to \mathcal{B}\left( \mathcal{H} \right)$ be a unital positive linear map and let $A\in \mathcal{B}\left( \mathcal{H} \right)$. If $f:[0,\infty)\to [0,\infty)$ is an  increasing operator convex function, then
\[f\left( {{w}^{2}}\left( \Phi \left( A \right) \right) \right)\le \left\| \Phi \left( \int_{0}^{1}{f\left( t{{\left| A \right|}^{2}}+\left( 1-t \right){{\left| {{A}^{*}} \right|}^{2}} \right)dt} \right) \right\|\le \frac{1}{2}\left\| \Phi \left( f\left( {{\left| A \right|}^{2}} \right)+f\left( {{\left| {{A}^{*}} \right|}^{2}} \right) \right) \right\|.\]
In particular, for any $1\le r\le 2$
\[{{w}^{2r}}\left( A \right)\le \left\| \int_{0}^{1}{{{\left( t{{\left| A \right|}^{2}}+\left( 1-t \right){{\left| {{A}^{*}} \right|}^{2}} \right)}^{r}}dt} \right\|\le \frac{1}{2}\left\| {{\left| A \right|}^{2r}}+{{\left| {{A}^{*}} \right|}^{2r}} \right\|.\]
\end{proposition}
\begin{proof}
We have
\[\begin{aligned}
 f\left( \left\langle \Phi \left( t{{\left| A \right|}^{2}}+\left( 1-t \right){{\left| {{A}^{*}} \right|}^{2}} \right)x,x \right\rangle  \right)&\le \left\langle \Phi \left( f\left( t{{\left| A \right|}^{2}}+\left( 1-t \right){{\left| {{A}^{*}} \right|}^{2}} \right) \right)x,x \right\rangle  \\ 
& \le \left\langle \Phi \left( tf\left( {{\left| A \right|}^{2}} \right)+\left( 1-t \right)f\left( {{\left| {{A}^{*}} \right|}^{2}} \right) \right)x,x \right\rangle  \\ 
& =t\left\langle \Phi \left( f\left( {{\left| A \right|}^{2}} \right) \right)x,x \right\rangle +\left( 1-t \right)\left\langle \Phi \left( f\left( {{\left| {{A}^{*}} \right|}^{2}} \right) \right)x,x \right\rangle   
\end{aligned}\]
for any unit vector $x\in \mathcal{H}$. Now, integrating over $t$ on $\left[ 0,1 \right]$,
\[\begin{aligned}
 \int_{0}^{1}{f\left( \left\langle \Phi \left( t{{\left| A \right|}^{2}}+\left( 1-t \right){{\left| {{A}^{*}} \right|}^{2}} \right)x,x \right\rangle  \right)dt}&\le \left\langle \Phi \left( \int_{0}^{1}{f\left( t{{\left| A \right|}^{2}}+\left( 1-t \right){{\left| {{A}^{*}} \right|}^{2}} \right)dt} \right)x,x \right\rangle  \\ 
& \le \frac{1}{2}\left\langle \Phi \left( f\left( {{\left| A \right|}^{2}} \right)+f\left( {{\left| {{A}^{*}} \right|}^{2}} \right) \right)x,x \right\rangle  \\ 
& \le \frac{1}{2}\left\| \Phi \left( f\left( {{\left| A \right|}^{2}} \right)+f\left( {{\left| {{A}^{*}} \right|}^{2}} \right) \right) \right\|.  
\end{aligned}\]
By taking supremum we get
\[\begin{aligned}
 \underset{\left\| x \right\|=1}{\mathop{\underset{x\in \mathcal{H}}{\mathop{\sup }}\,}}\,\left( \int_{0}^{1}{f\left( \left\langle \Phi \left( t{{\left| A \right|}^{2}}+\left( 1-t \right){{\left| {{A}^{*}} \right|}^{2}} \right)x,x \right\rangle  \right)dt} \right)&\le \left\| \Phi \left( \int_{0}^{1}{f\left( t{{\left| A \right|}^{2}}+\left( 1-t \right){{\left| {{A}^{*}} \right|}^{2}} \right)dt} \right) \right\| \\ 
& \le \frac{1}{2}\left\| \Phi \left( f\left( {{\left| A \right|}^{2}} \right)+f\left( {{\left| {{A}^{*}} \right|}^{2}} \right) \right) \right\|.  
\end{aligned}\]
Taking into account the first inequality in Theorem \ref{7}, we get the desired result.
\end{proof}

In the next result, we present a two-operator version. In \cite[Theorem 1]{sab}, part of this result was shown under the additional condition that $f$ is geometrically convex. In the next result, we remove this condition and we still obtain a better estimate. 
\begin{theorem}
	Let $A\in \mathcal{B}\left( \mathcal{H} \right)$, $f:[0,\infty)\to [0,\infty)$ be an increasing convex function. Then
\[f\left( w\left( {{B}^{*}}A \right) \right)\le \underset{\left\| x \right\|=1}{\mathop{\underset{x\in \mathcal{H}}{\mathop{\sup }}\,}}\,\left( \int_{0}^{1}{f\left( {{\left\| {{\left( t{{\left| A \right|}^{2}}+\left( 1-t \right){{\left| B \right|}^{2}} \right)}^{\frac{1}{2}}}x \right\|}^{2}} \right)dt} \right)\le \frac{1}{2}\left\| f\left( {{\left| A \right|}^{2}} \right)+f\left( {{\left| B \right|}^{2}} \right) \right\|.\]
\end{theorem}
\begin{proof}
In \eqref{13}, let $a=\left\langle {{\left|  A  \right|}^{2}}x,x \right\rangle$ and $b=\left\langle {{\left|  {{B}^{*}}  \right|}^{2}}x,x \right\rangle$,   where $x\in \mathcal{H}$ is a unit vector. Then
\begin{equation}\label{12}
\begin{aligned}
 f\left( \frac{\left\langle {{\left|  A  \right|}^{2}}x,x \right\rangle +\left\langle {{\left|  B  \right|}^{2}}x,x \right\rangle }{2} \right)&\le \int_{0}^{1}{f\left( \left\langle \left(t{{\left|  A  \right|}^{2}}+\left( 1-t \right){{\left|  B  \right|}^{2}}\right)x,x \right\rangle  \right)dt} \\ 
& \le \frac{f\left( \left\langle {{\left|  A  \right|}^{2}}x,x \right\rangle  \right)+f\left( \left\langle {{\left|  B  \right|}^{2}}x,x \right\rangle  \right)}{2}.  
\end{aligned}
\end{equation}
Utilizing arithmetic-geometric mean inequality and Schwarz inequality, we have
\begin{equation}\label{10}
\begin{aligned}
 f\left( \left| \left\langle  {{B}^{*}}  A x,x \right\rangle  \right| \right)& =f\left( \left| \left\langle  A x, B x \right\rangle  \right| \right) \\ 
& \le f\left( \left\|  A x \right\|\left\|  B x \right\| \right) \\ 
& =f\left( \sqrt{\left\langle {{\left|  A  \right|}^{2}}x,x \right\rangle \left\langle {{\left|  B  \right|}^{2}}x,x \right\rangle } \right) \\ 
& \le f\left( \frac{\left\langle {{\left|  A  \right|}^{2}}x,x \right\rangle +\left\langle {{\left|  B  \right|}^{2}}x,x \right\rangle }{2} \right).  
\end{aligned}
\end{equation}
On the other hand, \eqref{conv_inner_phi_intro} implies
\begin{equation}\label{11}
\frac{f\left( \left\langle {{\left| A  \right|}^{2}}x,x \right\rangle  \right)+f\left( \left\langle {{\left|  B  \right|}^{2}}x,x \right\rangle  \right)}{2}\le \frac{1}{2}\left\langle \left(f\left( {{\left|  A  \right|}^{2}} \right)+f\left( {{\left|  B  \right|}^{2}} \right)\right)x,x \right\rangle.
\end{equation}
Combining  \eqref{12}, \eqref{10}, and \eqref{11}, we get
\[\begin{aligned}
 f\left( \left| \left\langle  {{B}^{*}}  A x,x \right\rangle  \right| \right)&\le \int_{0}^{1}{f\left( \left\langle \left(t{{\left|  A  \right|}^{2}}+\left( 1-t \right){{\left|  B  \right|}^{2}}\right)x,x \right\rangle  \right)dt} \\ 
& \le \frac{1}{2}\left\langle \left(f\left( {{\left|  A  \right|}^{2}} \right)+f\left( {{\left|  B  \right|}^{2}} \right)\right)x,x \right\rangle   
\end{aligned}\]
for any unit vector $x\in \mathcal{H}$. Taking the supremum over $x$ implies the desired inequalities.
\end{proof}

Using similar argument as in Proposition \ref{8}, we get the following result.
\begin{proposition}
	Let $A\in \mathcal{B}\left( \mathcal{H} \right)$, $f:[0,\infty)\to [0,\infty)$ be an increasing operator convex function. Then
	\[f\left( w\left( {{B}^{*}}A \right) \right)\le \left\| \int_{0}^{1}{f\left( t{{\left| A \right|}^{2}}+\left( 1-t \right){{\left| B \right|}^{2}} \right)dt} \right\|\le \frac{1}{2}\left\| f\left( {{\left| A \right|}^{2}} \right)+f\left( {{\left| B \right|}^{2}} \right) \right\|.\]
In particular, for any $1\le r\le 2$	
	\[{{w}^{r}}\left( {{B}^{*}}A \right)\le \left\| \int_{0}^{1}{{{\left( t{{\left| A \right|}^{2}}+\left( 1-t \right){{\left| B \right|}^{2}} \right)}^{r}}dt} \right\|\le \frac{1}{2}\left\| {{\left| A \right|}^{2r}}+{{\left| B \right|}^{2r}} \right\|.\]
\end{proposition}

{\tiny \vskip 0.3 true cm }

{\tiny (M. Sababheh) Department of Basic Sciences, Princess Sumaya University For Technology, Al Jubaiha, Amman 11941, Jordan.}

{\tiny \textit{E-mail address:} sababheh@psut.edu.jo}

{\tiny \vskip 0.3 true cm }

{\tiny (H.R. Moradi) Department of Mathematics, Payame Noor University (PNU), P.O. Box 19395-4697, Tehran, Iran.}

{\tiny \textit{E-mail address:} hrmoradi@mshdiau.ac.ir }

\end{document}